\documentclass{amsart}

\newtheorem{theorem}[equation]{Theorem}
\newtheorem{lemma}[equation]{Lemma}
\newtheorem{corollary}[equation]{Corollary}
\newtheorem{proposition}[equation]{Proposition}

\theoremstyle{definition}

\newtheorem*{example*}{Example}
\newtheorem{example}[equation]{Example}
\newtheorem{examples}[equation]{Examples}
\newtheorem{remark}[equation]{Remark}
\newtheorem*{remark*}{Remark}

\newcommand{\bZ}{{\mathbb Z}}

\newcommand{\calL}{{\mathcal L}}

\newcommand{\wotimes}{\widehat{\otimes}}

 \DeclareMathOperator{\End}{End}

\newenvironment{romanenumerate}
 {\begin{enumerate}
 
 }{\end{enumerate}}

\title{The existence of superinvolutions}
\author{A. Elduque\and O. Villa}
\address{Departamento de Matem\'aticas e Instituto Universitario
de Matem\'aticas y Aplicaciones, Universidad de Zaragoza, 50009
Zaragoza, Spain} \email{elduque@unizar.es}
\address{Departamento de Matem\'aticas, Universidad de Zaragoza,
Pza. San Francisco s/n, 50009 Zaragoza, Spain}
\email{oliver.villa@unizar.es}
\thanks{The first author has been supported by the Spanish Ministerio de
Educaci\'{o}n y Ciencia and FEDER (MTM 2004-081159-C04-02) and by the
Diputaci\'on General de Arag\'on (Grupo de Investigaci\'on de
\'Algebra). The second author was funded by an ``Ayuda a la
Investigaci\'on del Vicerrectorado de Investigaci\'on de la
Universidad de Zaragoza'', and wants to thank D.W.~Lewis for his
hospitality in September 2006.}
\begin{document}

\begin{abstract}
Superinvolutions on graded associative algebras constitute a source
of Lie and Jordan superalgebras. Graded versions of the classical
Albert and Albert-Riehm Theorems on the existence of
superinvolutions are proven. Surprisingly, the existence of
superinvolutions of the first kind is a rare phenomenon, as
nontrivial central division superalgebras are never endowed with
this kind of superinvolutions.
\end{abstract}

\maketitle

\section{Introduction}

\noindent Albert's Theorem (see \cite{alb,scharl}) asserts that a
finite dimensional central simple algebra has an involution of the
first kind if and only if the order of its class in the Brauer group
is at most $2$, while Albert-Riehm Theorem (see \cite{riehm,scharl})
asserts that a necessary and sufficient condition for the existence
of an involution of the second kind is that the so called
corestriction is trivial in the Brauer group.

Given an associative superalgebra (that is, a graded associative
algebra) $A=A_0\oplus A_1$, a superinvolution is a graded linear map
$\xi:A\rightarrow A$, which is a superantiautomorphism (that is,
$\xi(xy)=(-1)^{xy}\xi(y)\xi(x)$ for any homogeneous elements $x,y\in
A$) and such that $\xi^2$ is the identity map. The set of skew
elements of a superinvolution $\{x\in A:\xi(x)=-x\}$ is a Lie
superalgebra under the graded bracket: $[x,y]=xy-(-1)^{xy}yx$, while
the set of fixed elements $\{x\in A: \xi(x)=x\}$ is a Jordan
superalgebra under the supersymmetrized product: $x\circ
y=xy+(-1)^{xy}yx$. Many of the classical simple Lie and Jordan
superalgebras (see \cite{kac,raczel}) arise in this way.
Superinvolutions on primitive associative superalgebras were studied
by Racine \cite{racine} and used in the classification of the simple
Jordan superalgebras \cite{raczel}.

However, up to our knowledge, no attempt has been made to obtain
necessary and sufficient conditions for a finite dimensional central
simple associative superalgebra to be endowed with a
superinvolution, thus obtaining analogues in the graded setting of
the classical Albert and Albert-Riehm Theorems. This is the purpose
of this paper.

Surprisingly, superinvolutions of the first kind are more difficult
to deal with, and it turns out that the natural analogue of Albert
Theorem is false: even if the class in the Brauer-Wall group of a
central simple algebra has order at most $2$, the superalgebra may
have no superinvolution of the first kind. Indeed, central simple
superalgebras of odd type never have superinvolutions of the first
kind (Theorem \ref{th:oddfirstkind}), and the same happens for
central division superalgebras of even type with nontrivial odd part
(Lemma \ref{evendivision} and Theorem \ref{th:evenfirstkind}).
However, Albert Theorem has a graded version, where superinvolutions
are substituted by superantiautomorphisms whose square is the
grading automorphism (Theorem \ref{albertgraded}).

The situation is nicer for superinvolutions of the second kind. For
these, the natural graded version of the classical Albert-Riehm
Theorem holds.

\smallskip

The paper is organized as follows. The next section is intended to
recall the basic definitions and results on associative
superalgebras. The basic sources are Lam's book \cite{lam} and
Racine's paper \cite{racine}. Some results will be proved in
slightly different ways, useful for our purposes. Then Section 3
will deal with superinvolutions of the first kind on central simple
superalgebras. It will be shown that the existence of such
superinvolutions is severely restricted. Section 4 will be devoted
to prove the above mentioned graded version of the classical Albert
Theorem, while superinvolutions of the second kind and the graded
version of Albert-Riehm Theorem will be the object of the last
section.

\section{Basic Concepts}

\subsection{Definitions and Notations}

We recall some basic definitions (compare with \cite[Chap.
IV]{lam}). Let $F$ be a field of characteristic different from $2$.
(\emph{This restriction will be assumed without mention throughout
the paper}.) A \emph{superalgebra} (also called \emph{graded
algebra}) $A$ is a (associative) $F$-algebra given in the form
$A=A_0\oplus A_1$, $F=F\cdot 1 \subseteq A_0$ and $A_iA_j \subseteq
A_{i+j}$ (subscripts modulo 2). We observe that $A_0$ is a
subalgebra of $A$. The elements of $hA:=A_0 \cup A_1$ are called
\emph{homogeneous elements} of $A$. For $h\in hA-\{0\}$ we define
\emph{the degree} $\delta h$ of $h$ by $\delta h=i$ if $h \in A_i$
($i=0,1$). To simplify the notation, for
every homogeneous element $x$ we define $(-1)^x:=(-1)^{\delta x}$.\\
A subspace $S \subseteq A$ is called \emph{graded} if it is the
direct sum of the intersections $S_i:=S \cap A_i$. We define $h(S)=S
\cap h(A)$. The \emph{graded center} of the superalgebra $A$ is
$\widehat Z(A):=\{x\in hA\mid xh=(-1)^{xh}hx \ \forall h \in hA \}$.
We shall call $A$ a \emph{central superalgebra} over $F$ if
$\widehat Z(A)=F$. The superalgebra $A$ is said to be a \emph{simple
superalgebra} over $F$ if $A$ has no proper ($\neq 0, \neq A$)
graded twosided ideals. A finite dimensional simple superalgebra
which is also a central superalgebra is called \emph{central simple
superalgebra}
(CSS, CSGA in Lam's book, see \cite{lam}).\\
The ordinary center of $A$, i.e. $Z(A)=\{ x\in A \mid xa=ax \
\forall a \in A\}$, is a graded subalgebra. If $A$ is a CSS over
$F$, then $Z(A)=F\oplus Z_1$ ($Z_1 \subseteq A_1$). If $Z_1=0$, we
say that $A$ is of \emph{even type}. If $Z_1 \neq 0$, we say that
$A$ is of
\emph{odd type}.\\
We denote the \emph{graded tensor product} of two graded algebras
$A$ and $B$ by $A \wotimes B$. We denote the Brauer group of a field
$F$ with $B(F)$ (resp. the Brauer-Wall group with $BW(F)$). The
opposite (resp. superopposite) algebra of a central simple algebra
(resp. CSS) $A$ is denoted $A^{op}$ (resp. $A^s$; in Lam's book the
algebra $A^s$ is denoted by $A^*$, see \cite[p. 80, 99]{lam}).

\begin{examples}
We list some important examples of CSSs.
\begin{enumerate}
\item[(a)]
We denote with $(A)$ the algebra $A$ with the grading given by
$A_1=0$.
\item[(b)]
Let $A$, $B$ be CSSs. Then $A \wotimes B$ is a CSS.
\item[(c)]
Let $a \in F^{\times}$. The space $F \oplus Fu$ with the relation $u^2=a$
and the grading $\delta u=1$ defines a CSS denoted by $F\langle
\sqrt{a} \rangle$. We call it \emph{quadratic graded algebra}.
\item[(d)]
For $a,b \in F^{\times}$ we define the \emph{graded quaternion algebra} as
follows:\\ $\langle a,b\rangle:=F\langle\sqrt{a}\rangle \widehat
\otimes F\langle\sqrt{b}\rangle$ (see \cite[p. 87]{lam}).
\item[(e)]
Let $D$ be a central division algebra over a field $F$. The algebra
of \\$(n+m)\times (n+m)$-matrices $M_{n+m}(D)$ can be viewed as CSS
by taking the diagonal components $M_{n}(D)$ and $M_{m}(D)$ as the
even part and the off-diagonal components as the odd part.
\end{enumerate}
\end{examples}

\begin{remark}
We observe that the definitons in examples $(c)$ and $(d)$ cover all
nontrivial gradings on a quadratic algebra and on a quaternion
algebra. For example, given a quaternion algebra with nontrivial
grading, the odd elements are orthogonal to $1$ with respect to the
(quaternionic) norm and hence they have zero trace. Since the
restriction of the norm to the space of odd elements is not
degenerate, then the algebra is forced to be of the form $\langle
a,b\rangle$ for some $a,b \in F^{\times}$.
\end{remark}

\noindent For completeness, we recall the structure theorems for
CSSs (see \cite[Chap. IV, 3.6 and 3.8]{lam}).
\begin{theorem}\label{structureodd}
Let $A$ be a CSS of odd type. Then:
\begin{enumerate}
\item[(1)]
$Z(A)=F\oplus Fz$, where $z\in Z_1$ and $z^2=a \in F^{\times}$. The
square class of $a$ does not depend on the choice of $z\in
Z_1-\{0\}$, and $Z(A)\simeq F\langle \sqrt{a} \rangle$ as graded
algebras.
\item[(2)]
There are graded algebra isomorphisms $$A\simeq (A_0)\wotimes
F\langle \sqrt{a} \rangle\simeq (A_0)\otimes F\langle \sqrt{a}
\rangle.$$
\item[(3)]
If $a \notin F^{\times2}$, then $A$ is a central simple algebra over
$Z(A)\simeq F(\sqrt{a})$.\\ If $a\in F^{\times2}$, then $Z(A)\simeq
F \times F$, and $A \simeq A_0 \times A_0$ (as ungraded algebra).
\end{enumerate}
In any case, $A$ is a semisimple separable $F$-algebra.
\end{theorem}

\begin{theorem} \label{structureeven}
Let $A$ be a CSS of even type, $A_1\neq 0$. Suppose $A$ is
isomorphic, as ungraded algebra, to $M_n(D)$, the central simple
$F$-algebra of $n\times n$-matrices with entries in the central
division algebra $D$ over $F$. Then:
\begin{enumerate}
\item[(1)]
There exists an element $z \in Z(A_0)$ such that $Z(A_0)=F\oplus Fz$
and $z^2=a \in F^{\times}$. The element $z$ is determined up to a
scalar multiple by these properties, and hence the square class of
$a$ is uniquely determined.
\item[(2)]
Suppose $a\in F^{\times2}$. Then $Z(A_0)\simeq F\times F$ and
$A\simeq M_{r+s}(F) \wotimes (D)$ with $r+s=n$. Moreover, $A_0\simeq
M_{r}(D) \times M_{s}(D)$.
\item[(3)]
Suppose $a\notin F^{\times2}$, and the field $Z(A_0)\simeq
F(\sqrt{a})$ can be embedded into $D$. Then there exists a grading
on $D$ such that $A\simeq (M_n(F))\wotimes D$. In this case,
$A_0\simeq M_n(D_0)$ is a central simple algebra over $Z(A_0)$.
\item[(4)]
Suppose $a\notin F^{\times2}$, and the field $Z(A_0)\simeq
F(\sqrt{a})$ can not be embedded into $D$. Then $n=2m$ is even, and
$A\simeq (M_m(D))\wotimes\langle -a,1\rangle$ as graded algebras. In
this case, $A_0\simeq M_m(D)\wotimes F(\sqrt{a})$ is a central
simple algebra over $Z(A_0)$.
\end{enumerate}
In any case, $A_0$ is a semisimple separable $F$-algebra.
\end{theorem}

\noindent We call an algebra \emph{central division superalgebra}
(CDS) if it is a CSS where every non-zero homogeneous element is
invertible.

\begin{remark}\label{superdiv}
If $A$ is a CSS of even type ($A_1\neq 0$, $Z(A_0)=F\oplus Fz$ and
$z^2=a \in F^{\times}$) with $a\notin F^{\times2}$ (i.e. in the cases $(3)$
and $(4)$ of the Theorem above), then we may write it as
$$A\simeq (M_n(F))\wotimes \Delta$$
where $\Delta$ is a CDS (see for example theorem 2 in
\cite{cristina} or the proof of
\cite[Chap. IV, 3.8]{lam}).\\
On the other hand, if $A$ is a CSS of odd type, then it is of the
form $(M_k(F))\wotimes \Delta$, where $\Delta$ is a CDS of odd type.
\end{remark}
\noindent Let $A=A_0+A_1$ be a superalgebra. For all $x\in A_0$,
$y\in A_1$ we define $\nu (x+y)=x-y$. The induced map $\nu$ is a
graded automorphism of $A$ called the \emph{grading automorphism}
(the main involution in Lam's book, see \cite[Chap. IV, Definition
3.7]{lam}). If $A$ is a CSS of even type with $Z(A_0)=F1+Fz$ and
$z^2\in F^\times$, then recall that $\nu(x)=zxz^{-1}$; in
particular, we have $uz=-zu$ for all $u \in A_1$. A
\emph{superantiautomorphism} of a superalgebra $A$ is a graded
additive map $\sigma:A\rightarrow A$ such that for all $a_\alpha \in
A_\alpha$ and $b_\beta \in A_\beta$ $\ \sigma(a_\alpha
b_\beta)=(-1)^{\alpha \beta}\sigma(b_\beta) \sigma(a_\alpha).$ We
call \emph{superinvolution} of $A$ a superantiautomorphism $\tau$
such that $\  \tau^2(x)=x$ for all $x\in A$. As for involutions, we
say that the superinvolution is \emph{of the first kind} if it is
$F$-linear, and \emph{of the second kind} otherwise.

\begin{remark}
Let $A$, $B$ be CSSs over $F$ with superinvolutions $\tau_A$ and
$\tau_B$. The map $\tau_A \otimes \tau_B$ is a superinvolution on $A
\wotimes B$.
\end{remark}

\noindent Before the study of involutions of the first kind, we
consider the case of the CSS $M_{n+m}(F)$: it shows that the
existence of a superantiautomorphism does not always imply the
existence of a superinvolution. However, here it is easy to see that
there is always a superantiautomorphism whose square is the grading
automorphism.

\begin{proposition}
The CSS $M_{n+m}(F)$ has always a superantiautomorphism $\varphi$
with $\varphi^2=\nu$. It has a superinvolution of the first kind if
and only if $n=m$ or $nm$ is even. \label{splitsuper}
\end{proposition}
\begin{proof}
As usual, we denote the inverse of a matrix $a$ with $a^t$. We
observe that the
map $$\varphi:M_{n+m}(F)\longrightarrow M_{n+m}(F),\qquad \left(%
\begin{array}{cc}
  a & b \\
  c & d \\
\end{array}%
\right)\mapsto \left(%
\begin{array}{cc}
  a^t & -c^t \\
  b^t & d^t \\
\end{array}%
\right)$$ is a superantiautomorphism. Moreover, one can check easily
that $\varphi^2=\nu$.\\
For the case $n=m$, it is enough to observe that the map
$$\tau:M_{n+n}(F)\longrightarrow M_{n+n}(F),\qquad \left(%
\begin{array}{cc}
  a & b \\
  c & d \\
\end{array}%
\right)\mapsto \left(%
\begin{array}{cc}
  d^t & -b^t \\
  c^t & a^t \\
\end{array}%
\right)$$ is a superinvolution. If $n\neq m$ and $\tau$ is a
superinvolution on $M_{n+m}(F)$, then $\tau$ is adjoint to a
superform (see \cite[Theorem 7]{racine} or Theorem
\ref{th:primesuperinvolhermitian}). Such a superform can be defined
if and only if $n$ or $m$ is even.
\end{proof}

\subsection{The graded Skolem-Noether theorem}
\noindent In this section we give a version of the Skolem-Noether
Theorem for superalgebras (see also Lemma $1$ in \cite{riehm}).

\noindent For any homogeneous and invertible element of a
superalgebra $A$, consider the \emph{inner automorphism} $\iota_a$
given by
\[
\iota_a(x)=(-1)^{ax}axa^{-1}
\]
for any homogeneous $x\in A$. We denote the disjoint union with
$\sqcup$. In the next proposition we describe the set
$\mathrm{Autg}(A)$ of graded $F$-linear automorphisms of $A$.

\begin{proposition}
Let $A$ be a CSS. Then $\mathrm{Autg}(A)=\iota_{A_0^\times} \sqcup
\iota_{A_1^\times}$
\end{proposition}

\begin{proof}
Let $\varphi$ be an element of $\mathrm{Autg}(A)$. We can write
$A=\End( _{\Delta}V)$, where $V$ is a graded vector space with the
action of the CDS $\Delta$ on the left. We assume that $A_1\neq 0$
(i.e. $V_1\neq 0$). Let $\Psi: \Delta^s\wotimes A \rightarrow
\End_{F}(V)$, defined by $d \otimes a \mapsto (\Psi(d \otimes a): v
\mapsto (-1)^{dv}dva).$ By dimension count, $\Psi$ is an isomorphism
of CSS. There are two $\Delta^s\wotimes A$-module structures on $V$,
namely $v.(d\otimes a)=(-1)^{dv}dva$ and $v\diamond (d\otimes
a)=(-1)^{dv}dva^\varphi$. But the only irreducible
$\End_F(V)$-modules of $\End_F(V)$ are $V$ and $V^s$, the
superopposite module of $V$, i.e $V$ with the roles
of $V_0$ and $V_1$ interchanged.\\
Two cases may occur.\\
\textbf{1.} There exists a $F$-linear even isomorphism $\sigma:V
\rightarrow V$ such that $v.(d\otimes
a)\sigma=(v\sigma)\diamond(d\otimes a)$
for all $d \in \Delta, v \in V, a \in A$.\\
Setting $a=1$ we obtain that $\sigma \in A_0$. Then, setting $d=1$,
we get $\varphi=\iota_{\sigma}$.\\
\textbf{2.} There exists a $F$-linear odd isomorphism $\sigma:V
\rightarrow V$ such that $v.(d\otimes
a)\sigma=(-1)^{a}(-1)^{d}(v\sigma)\diamond(a\otimes d)$ for all $d
\in \Delta, v \in V, a \in A$. Setting again, as before, $a=1$ and
$d=1$
we get $\varphi= \iota_{\sigma}$.
\end{proof}

\begin{remark} If, in the Proposition above, $A$ is even, then the
grading automorphism is $\nu=\iota_z$, where $z$ is any element such
that $Z(A_0)=F1+Fz$ and $z^2\in F^\times$. If $A$ is odd, then
$Z(A)=F1+Fz$ for an invertible odd element $z$, and again
$\nu=\iota_z$. Since here $A_1=A_0z$, it follows that
$\iota_{A_1^\times}=\nu\circ\iota_{A_0^\times}$.
\end{remark}

\begin{remark}
If $\varphi \in \mathrm{Autg}(A)$ fixes $Z(A)$ and $Z(A_0)$
elementwise, then $\varphi \in \iota_{A_0^\times}$.
\end{remark}

\subsection{The graded Jacobson density Theorem and superinvolutions}

The results of Michel Racine (see \cite{racine}) for prime
superalgebras with superinvolution will be proven in a slightly
different way. \noindent In this context we use words prime and
semiprime in the obvious graded sense. We observe that every simple
superalgebra is prime.

\begin{lemma}\label{le:Brauer} Let $A$ be a semiprime
superalgebra. Then
\begin{romanenumerate}
\item \emph{(Brauer)} If $I$ is a minimal right ideal of $A$, then there is an
idempotent $e\in I_0$ such that $I=eA$. Moreover, for any
homogeneous element $x\in I$ with $xI\ne 0$, there exists an
idempotent $e=e^2\in I$ such that $I=eA$ and $ex=xe=x$.
\item If $e$ is a nonzero idempotent of $A_0$ and $eA=I$ is a minimal
right ideal of $A$, then $eAe$ is a division superalgebra, which is
isomorphic to the centralizer superalgebra $\End_A(I)$.
\item If $e$ is a nonzero idempotent of $A_0$ such that $eAe$ is a
division superalgebra, then $eA$ is a minimal right ideal of $A$.
\item If $a$ is an homogeneous element of $A$ such that $aA$ is a
minimal right ideal of $A$, then $Aa$ is a minimal left ideal of
$A$.
\end{romanenumerate}
\end{lemma}
\begin{proof}
For (i), note that $I^2\ne 0$ by semiprimeness, so $I^2=I$ by
minimality and there is a nonzero homogeneous element $x\in I_0\cup
I_1$ such that $xI\ne 0$. Again, since $I$ is minimal, $I=xI$, and
hence there is an element $e\in I_0$ such that $x=xe$. Take $J=\{
r\in I: xr=0\}$. Then $J$ is a right ideal of $A$ strictly contained
in $I$, so $J=0$. Since $e^2-e\in J$, we conclude that $e$ is a
nonzero idempotent of $I_0$, and since $0\ne eI\subseteq I$, the
minimality of $I$ forces $I=eI$, as desired. In particular $ey=y$
for any $y\in I$, so the last assertion follows.

\noindent For (ii)
notice that if $x$ is an homogeneous element of $A$ such that
$exe\ne 0$, then $0\ne exeA\subseteq eA$ so $exeA=eA$ by minimality.
Therefore there exists an homogeneous element $y\in A$ such that
$exey=e$, so $(exe)(eye)=e$, which is the unity of the superalgebra
$eAe$. Therefore, any nonzero homogeneous element of $eAe$ has a
right inverse. This is enough to ensure that $eAe$ is a division
superalgebra. Besides, the linear map $eAe\rightarrow \End_A(I)$
given by $exe\mapsto \rho_{exe}:I\rightarrow I$, such that
$\rho_{exe}(z)=exez$ for any $z\in I$ is easily shown to be an
isomorphism. Note that this is valid even if $A$ is not semiprime.

\noindent Now assume that $0\ne e=e^2\in A_0$ satisfies that $eAe$
is a division superalgebra, and let $I$ be a nonzero right ideal
contained in $eA$. Let $x$ be a nonzero homogeneous element of $I$,
so $x=ex$. If $exAe$ were $0$, then $(AexA)^2$ would be $0$ too,
contradicting the semiprimeness of $A$. Therefore there is an
homogeneous element $y\in A$ such that $exye\ne 0$, and since $eAe$
is a division superalgebra, there is another homogeneous element $z$
such that $xyeze=(exye)(eze)=e$. In particular, $e\in xA$ and
$eA\subseteq xA\subseteq I$, so $I=eA$. This shows that $ea$ is a
minimal right ideal.

\noindent Finally, assume that $a$ is an homogeneous element such
that $aA$ is a minimal right ideal. As in (i), let $e$ be an
idempotent such that $ae=ea=a$ and $aA=eA$. Because of (ii), $eAe$
is a division superalgebra, and by symmetry, item (iii) shows that
$Ae$ is a minimal left ideal. But $Aa\ne 0$ by semiprimeness, and
$Aa=(Ae)a$ is a homomorphic image of the irreducible left module
$Ae$, so it is irreducible too. That is, $Aa$ is a minimal left
ideal.
\end{proof}

\noindent The idempotents $e$ such that $eAe$ is a division
superalgebra will be called \emph{primitive idempotents}. Given a
superinvolution $*$ in an associative superalgebra, $H(A,*)$ and
$S(A,*)$ will denote, respectively, the set of fixed elements by $*$
and the set of elements $x\in A$ such that $x^*=-x$.

\begin{theorem}\label{th:primesupercases}
Let $A$ be a prime superalgebra with minimal right ideals, and let
$*$ be a superinvolution of $A$. Then one and only one of the
following situations occurs:
\begin{romanenumerate}
\item There exists a primitive idempotent such that $e^*=e$,
\item $A_1=0$ and there exists a primitive idempotent such that
$eH(A,*)e^*=0$. In this case $eAe$ is a field and there are elements
$u\in eAe^*$ and $v\in e^*Ae$ such that $u,v\in S(A,*)$, $uv=e$ and
$vu=e^*$.
\item There exists a primitive idempotent such that $eA_0 e^*=0$.
In this case, this idempotent $e$ can be taken satisfying that there
are elements $u\in eA_1 e^*$ with $u^*=u$ and $v\in e^*A_1 e$ with
$v^*=-v$, such that $uv=e$ and $vu=e^*$.
\end{romanenumerate}
\end{theorem}
\begin{proof}
Assume that there is a minimal right ideal $I$ and an homogeneous
element $a\in I$ such that $aa^*I\ne 0$. Then take $x=aa^*\,(\in I)$
and note that $x^*=(-1)^xx$. By Lemma \ref{le:Brauer} there is a
primitive idempotent $e\in I$ with $I=eA$ and $xe=ex=x$. Therefore,
$xee^*=xe^*=(-1)^xx^*e^*=(-1)^x(ex)^*=(-1)^xx^*=x$. Then, as in the
proof of Lemma \ref{le:Brauer}, $f=ee^*$ is an idempotent, with
$f^*=f$ and $I=fA$ and case (i) appears.

\noindent Otherwise, for any minimal right ideal $I$ of $A$, and any
homogeneous element $a\in I$, $aa^*I=0$ holds. Take a minimal right
ideal $I$ of $A$ and assume that there exists an homogeneous element
$a\in I$ such that $aa^*\ne 0$. By minimality, $I=aa^*A= aA$, and
$I^*I=Aaa^*aa^*A\subseteq A(aa^*I)=0$. By Lemma \ref{le:Brauer}(iv),
$Aa$ is a minimal left ideal, and hence $a^*A=(Aa)^*$ is a minimal
right ideal. Take $J=a^*A$. If there were an homogeneous element $x$
in $J$ with $xx^*\ne 0$, as before we would have $0=J^*J=Aaa^*A$,
but this is impossible since $A$ is semiprime.

\noindent Hence, either the situation in (i) holds or there is a
minimal right ideal $I$ of $A$ such that
\begin{equation}\label{eq:xx*0}
xx^*=0\quad\text{for any homogeneous element $x\in I$.}
\end{equation}
(Notice that up to now the arguments are valid assuming only that
$A$ is semiprime.)

\medskip

\noindent Let $I$ be such a minimal right ideal. By Lemma
\ref{le:Brauer}, $I=eA$ for a primitive idempotent $e$. Hence, for
any homogeneous element $x\in A$, $exx^*e^*=(ex)(ex)^*=0$. In
particular, for any $x\in A_0$, $e(e+x)(e+x)^*e^*=0=ee^*=exx^*e^*$,
so $e(x+x^*)e^*=0$ (that is, $eH(A_0,*)e^*=0$), or
\begin{equation}\label{eq:exe*skew}
(ex_0 e^*)^*=-ex_0 e^*
\end{equation}
for any $x_0\in A_0$. (Note that if $A_1=0$, this condition is
equivalent to the condition in \eqref{eq:xx*0}.)

\noindent Assume first that $eA_0 e^*\ne 0$, and take $z\in A_0$
such that $eze^*\ne 0$. By primeness, $eze^*Ae\ne 0$, and since
$eAe$ is a division superalgebra, we can obtain easily another
element $t\in A_0$ such that $eze^*te=e$. Take $u=eze^*$ and
$v=e^*te$, so $uv=e$. Besides, $u^*=-u$ holds by
\eqref{eq:exe*skew}.

\noindent Now, $v^2=e^*t(ee^*)te=0$ (by \eqref{eq:xx*0}),
$u^2=-uu^*=0$, as $u\in I_0$, $e^*e=(uv)^*uv=v^*u^*uv=-v^*u^2v=0$,
and $v=e^*v=(uv)^*v=v^*u^*v=-v^*uv=-v^*e=-v^*$, so
$e^*=v^*u^*=(-v)(-u)=vu$. Let us denote by $\Delta$ the division
superalgebra $eAe$. Consider the linear map $\Delta\rightarrow
\Delta:$ $d\mapsto \bar d=ud^*v$. Note that for any homogeneous
$d,d_1,d_2\in \Delta$:
\[
\begin{split}
&\bar{\bar d}=u(ud^*v)^*v=uv^*du^*v=uvduv=ede=d,\\[6pt]
&\overline{d_1d_2}=u(d_1d_2)^*v=(-1)^{d_1d_2}ud_2^*d_1^*v
   =(-1)^{d_1d_2}ud_2^*e^*d_1^*v\\
   &\phantom{\overline{d_1d_2}}
   =(-1)^{d_1d_2}ud_2^*vud_1^*v=(-1)^{d_1d_2}\bar d_2\bar d_1.
\end{split}
\]
Therefore, this map is a superinvolution. But for any
$d\in\Delta_0$,
\[
\bar d=ud^*v=eud^*e^*v=-(eud^*e^*)^*v=-edu^*e^*v=eduv=ede=d,
\]
where we have used \eqref{eq:exe*skew}, together with the fact that
$u^*=-u$ and $uv=e$. Therefore the restriction of the
superinvolution $d\mapsto \bar d$ is the identity, and since this is
an ordinary involution of the division algebra $\Delta_0$, we
conclude that $\Delta_0$ is a field. Besides, for any $d\in\Delta_1$
with $\bar d=\pm d$ (that is $d\in H(\Delta,-)_1\cup
S(\Delta,-)_1$), $d^2\in\Delta_0$, so $d^2=\overline{d^2}$. Hence
\[
d^2=\overline{d^2}=(-1)^{dd}\bar d^2=-d^2.
\]
Thus $d^2=0$, and since $\Delta$ is a division superalgebra, $d=0$.
Hence $\Delta_1=H(\Delta,-)_1\oplus S(\Delta,-)_1=0$, and
$\Delta=\Delta_0$ is a field. But for any $x,y\in A_1$,
$e(x+y)(x+y)^*e^*=0=exx^*e^*=0yy^*e^*$ by \eqref{eq:xx*0}, so
$exy^*e^*=-eyx^*e^*$. On the other hand, \eqref{eq:exe*skew} shows
that $(exy^*e^*)^*=-exy^*e^*$, that is, $-eyx^*e^*=-exy^*e^*$. We
conclude that $exy^*e^*=0$ for any $x,y\in A_1$. Therefore, $eA_1
A_1 e^*=0$, so $eA_1 A_1 v=eA_1 A_1 e^*v=0$. But $eA_1 A_0
v\subseteq eA_1 e=\Delta_1=0$. Hence $eA_1 Av=0$, which implies,
since $A$ is prime, that $eA_1 =0$. Now, $eAA_1=eA_0
A_1+(eA_1)A_1\subseteq eA_1+(eA_1)A_1=0$, and $A_1=0$ by primeness.
We are in case (ii) of the Theorem, since $eH(A,*)e^*$ is $0$
because of \eqref{eq:exe*skew}.

\medskip

\noindent Finally, assume that the minimal right ideal $I=eA$
satisfies \eqref{eq:xx*0}, but $eA_0 e^*=0$. Again, let $\Delta$ be
the division superalgebra $eAe$. Since $eAe^*=eA_1 e^*\ne 0$ and
$\Delta_1 eA_1 e^*\subseteq eA_0 e^*=0$, and the nonzero elements of
$\Delta_1$ are invertible, it follows that $\Delta_1=0$ in this
case.

\noindent If there exists an odd element $x\in A_1$ such that
$e(x+x^*)e^*\ne 0$, then there is an element $z=x+x^*=z^*\in A_1$
such that $eze^*\ne 0$. As before we find an element $t\in A_1$ such
that $eze^*te=e$ and take $u=eze^*$ and $v=e^*te$. Then
$u^*=ez^*e^*=u$, $uu^*=0=ee^*$ because of \eqref{eq:xx*0}, so
$u^2=0$. Also, 
$v=e^*v=(uv)^*v=-(v^*u^*)v=-v^*uv=-v^*e=-v$, and
$e^*=(uv)^*=-v^*u^*=vu$, thus obtaining the situation in item (iii)
of the Theorem.

\noindent Otherwise, for any $x\in A_1$ $e(x+x^*)e^*=0$, or
$(exe^*)^*=-(exe^*)$. As before there are $z,t\in A_1$ such that
$eze^*te=e$. Take $u=eze^*$ and $v=e^*te$, so that $e=uv$, but now
$u^*=-u$ and $v=e^*v=(uv)^*v=-v^*u^*v=v^*uv=v^*e=v^*$, and
$e^*=(uv)^*=-v^*u^*=vu$. Consider in this case the primitive
idempotent $f=e^*$. Then $fA_0 f^*=vuA_0 uv=ve(uA_0 u)e^*v\subseteq
v(eA_0 e^*)v=0$, so in particular the condition in \eqref{eq:xx*0}
holds trivially for the minimal right ideal $fA$, and
$f(v+v^*)f^*=2fvf^*=2v\ne 0$. Therefore, it is enough to change $e$
to $f$ to obtain the situation in item (iii).

\medskip

\noindent To finish the proof of the Theorem, it must be checked
that only one of the three possible situations occur. It is clear
that (ii) and (iii) are mutually exclusive, since $A_1 =0$ in (ii)
but not in (iii). Also, if $e$ is a primitive idempotent as in (i),
$I=eA$, and $f$ is a primitive idempotent such that the minimal
right ideal $J=fA$ satisfies the condition in \eqref{eq:xx*0}, then
by primeness $IJ\ne 0$, so by minimality $I=IJ$. Thus, there is an
homogeneous element $x\in I$ such that $0\ne xJ$ and $I=xJ$ by
minimality. Then $e=xz$ for some homogeneous element $z\in J$ of the
same parity as $x$. But $e=e^2=ee^*=(xz)(xz)^*=\pm xzz^*x^*$, which
is $0$ by \eqref{eq:xx*0}, a contradiction. Therefore, case (i) is
not compatible with cases (ii) or (iii).
\end{proof}

\medskip

\noindent Let $(\Delta,-)$ be a division superalgebra endowed with a
superinvolution, let $V$ be a left module over $\Delta$ and let
$h:V\times V\rightarrow \Delta$ be a $\epsilon$-hermitian
($\epsilon=\pm 1$) nondegenerate form of degree $l$ ($l\in\{ 0,
1\}$). That is, $h$ is biadditive, $h(V_i,V_j)\subseteq
\Delta_{i+j+l}$ for any $i,j\in\{ 0, 1\}$ and
\[
h(dx,y)=dh(x,y),\qquad h(y,x)=\epsilon(-1)^{xy}\overline{h(x,y)}
\]
for any homogeneous elements $x,y\in V_0\cup V_1$ and $d\in \Delta$.
This implies that $h(x,dy)=(-1)^{dy}h(x,y)\overline{d}$ for any
homogeneous $x,y\in V$ and $d\in \Delta$.

\noindent The case of $\Delta_1=0$, $\Delta_0$ a field, $-$ the
identity and $\epsilon=-1$ corresponds to the alternating bilinear
forms over a field. If $\epsilon=1$, $h$ will just be said to be an
hermitian form.

\noindent Consider the superalgebra with superinvolution $\calL(V)$
with
\begin{multline*}
\calL(V)_i=\{f\in \End_\Delta(V)_i: \exists f^*\in\End_\Delta(V)\
\textrm{such that}\\ h(xf,y)=(-1)^{fx}h(x,yf^*)\ \forall x\in
V,\forall y\in V_0\cup V_1\}.
\end{multline*}
Note that $f^*$ is unique by the nondegeneracy of $f$, which gives
the superinvolution in $\calL(V)$, and that the action of the
elements of $\End_\Delta(V)$ is written on the right.

\noindent For any homogeneous elements $v,w\in V$, consider the
$\Delta$-linear map $h_{v,w}$ given by:
\[
a\mapsto ah_{v,w}=h(a,v)w\,.
\]
For any homogeneous $x,y,v,w\in V$,
\[
\begin{split}
h(xh_{v,w},y)&=h(x,v)h(w,y)\\
 &=(-1)^{v(w+y)}h(x,\overline{h(w,y)}v)\\
 &=\epsilon(-1)^{v(w+y)+wy}h(x,g(y,w)v)\\
 &=\epsilon(-1)^{v(w+y)+wy}h(x,yh_{w,v}),
\end{split}
\]
so that
\begin{equation}\label{eq:hvw*}
(h_{v,w})^*=\epsilon(-1)^{vw}h_{w,v}
\end{equation}
for any homogeneous $v,w\in V$. Also, for any homogeneous
$f\in\calL(V)$:
\begin{equation}\label{eq:hvwf}
h_{v,w}f=h_{v,wf},\qquad fh_{v,w}=(-1)^{fv}h_{vf^*,w}.
\end{equation}
Therefore, the span $h_{V,V}$ of the $h_{v,w}$'s is an ideal of
$\calL(V)$ closed under the superinvolution, and it acts irreducibly
on $V$. Besides, for any homogeneous element $\psi$ in the
centralizer $\End_{h_{V,V}}(V)$ (action on the left), and any
homogeneous $x,v,w\in v$, $\psi(xh_{v,w})=\psi(x)h_{v,w}$, so that,
if $h(x,v)=0$, then also $h(\psi x,v)=0$. Since $h$ is
nondegenerate, there is an homogeneous element $d_x\in \Delta$ such
that $\psi(x)=d_x x$. If now we take $v$ with $h(x,v)=1$, then we
obtain that $\psi(w)=d_xw$ for any homogeneous $w$. That is, $\psi$
is the left multiplication by $d_x$. This shows that the centralizer
of the action of $h_{V,V}$ on $V$ is $\Delta$.

\begin{lemma}\label{le:Jacobson} Let $A$ be an
associative superalgebra and let $V$ be an irreducible right
$A$-module with centralizer $\Delta=\End_A(V)$ (action of $\Delta$
on the left), which is a division superalgebra by Schur's Lemma.
Then for any homogeneous $\Delta$-linearly independent elements
$v_1,\ldots,v_n$ there is an homogeneous element $a\in A$ such that
$v_1a\ne 0$, $v_2a=\cdots=v_na=0$.
\end{lemma}
\begin{proof}
This is proved by induction on $n$, the case $n=1$ being trivial.
Assuming the result proven for $n-1$, let $J=\{a\in A: v_3a=\cdots
v_na=0\}$ be the right annihilator of $v_3,\ldots,v_n$. By the
induction hypothesis $v_2J\ne 0 \ne v_1J$, and by irreducibility
$V=w_2J=w_1J$. If there is an homogeneous element $a\in J$ with
$w_2a=0\ne w_1 a$, we are done. Otherwise, the map $\psi:
V=w_2J\rightarrow V=w_1J$ such that $\psi(w_2a)=w_1a$ for any $a\in
J$ is well defined and belongs to the centralizer
$\End_A(V)=\Delta$, so that $\psi=d$ for some homogeneous element
$d\in \Delta$. But then $(w_1-dw_2)J=0$, so by the induction
hypothesis, $w_1-dw_2\in \Delta w_3+\cdots+\Delta v_n$, a
contradiction.
\end{proof}

\begin{corollary}\label{co:Jacobson} \textbf{\emph{(Jacobson's density)}}\quad
Let $A$ be an associative superalgebra and let $V$ be an irreducible
right $A$-module with centralizer $\Delta=\End_A(V)$ (action of
$\Delta$ on the left). Then for any homogeneous $\Delta$-linearly
independent elements $v_1,\ldots,v_n$ and for any elements $w_1,
\ldots,w_n$ there is an element $a\in A$ such that $v_ia=w_i$ for
$i=1,\ldots,n$.
\end{corollary}

\begin{corollary}\label{co:hVVsimple}
Let $(\Delta,-)$ be a division superalgebra endowed with a
superinvolution, and let $V$ be a left $\Delta$-module endowed with
a nondegenerate $\epsilon$-hermitian form $h:V\times V\rightarrow
\Delta$. Then $h_{V,V}$ is the only simple ideal of any subalgebra
of $\calL(V)$ containing it. In particular, any such subalgebra is
prime.
\end{corollary}
\begin{proof} If $f$ is a nonzero homogeneous element of $h_{V,V}$ and it is
written as $f=\sum_{i=1}^nh_{v_i,w_i}$ (for homogeneous $v_i,w_i$,
$i=1,\ldots ,n$), with minimal $n$, then $w_1,\ldots,w_n$ are
linearly independent over $\Delta$, so by Lemma \ref{le:Jacobson}
there is an homogeneous element $g\in h_{V,V}$ such that $\hat
w_1=w_1g\ne 0$ and $w_2g=\cdots w_ng=0$. Because of \eqref{eq:hvwf}
$fg=h_{v_1,\hat w_1}$, and using again \eqref{eq:hvwf} it follows
that $h_{V,V}$ is simple. Also, \eqref{eq:hvwf} shows that the ideal
generated by any element of a subalgebra of $\calL(V)$ containing
$h_{V,V}$ intersects $h_{V,V}$ nontrivially. Hence the result.
\end{proof}

\noindent The following result follows at once from \eqref{eq:hvwf}:

\begin{lemma}\label{le:hvV}
With the same hypotheses as in Corollary \ref{co:hVVsimple}, for any
nonzero homogeneous element $v\in V$, $h_{v,V}$ is a minimal right
ideal of any subalgebra of $\calL(V)$ containing $h_{V,V}$.
\end{lemma}

\noindent Note that, under the hypotheses of the Lemma, an
homogeneous element $w\in V$ can be taken with $h(w,v)=1$. Then, the
even element $e=h_{v,w}\in h_{V,V}$ is an idempotent with
$e^*=\epsilon(-1)^{vw}h_{w,v}$ because of \eqref{eq:hvwf} and
\eqref{eq:hvw*}.

\begin{theorem}\label{th:primesuperinvolhermitian}
Let $A$ be a prime superalgebra with minimal right ideals, and let
$*$ be a superinvolution of $A$. Then one and only one of the
following situations occurs:
\begin{romanenumerate}
\item There exists a division superalgebra with a superinvolution
$(\Delta,-)$, a left $\Delta$-module $V$ with $V_0\ne 0$ endowed
with a nondegenerate hermitian even superform $h:V\times V
\rightarrow \Delta$, and a faithful representation
$\rho:A\rightarrow \End_{\Delta}(V)$ such that $\rho(A)$ is a
subalgebra of $\calL(V)$, containing $h_{V,V}$,  closed under the
superinvolution of $\calL(V)$, and such that $\rho(a^*)=\rho(a)^*$
for any $a\in A$.
\item $A_1=0$, and there is a field $F$ and a vector space $V$ over
$F$ endowed with a nondegenerate alternating bilinear form
$h:V\times V\rightarrow F$ and a faithful representation
$\rho:A\rightarrow \End_F(V)$ such that $\rho(A)$ is a subalgebra of
$\calL(V)$, containing $h_{V,V}$,  closed under the superinvolution
of $\calL(V)$, and such that $\rho(a^*)=\rho(a)^*$ for any $a\in A$.
\item There exists a division algebra with an involution $(D,-)$, a
left $\bZ_2$-graded vector space $V$ endowed with a nondegenerate
hermitian odd form $h:V\times V\rightarrow D$ , and a faithful
representation $\rho:A\rightarrow \End_D(V)$ such that $\rho(A)$ is
a subalgebra of $\calL(V)$, containing $h_{V,V}$, closed under the
superinvolution of $\calL(V)$, and such that $\rho(a^*)=\rho(a)^*$
for any $a\in A$.
\end{romanenumerate}
\noindent Conversely, any such superalgebra is prime, contains
minimal right ideals and it is endowed with a superinvolution.
\end{theorem}
\begin{proof}
Let $A$ be a prime superalgebra with minimal right ideals, and let
$*$ be a superinvolution of $A$. According to Theorem
\ref{th:primesupercases} three possible situations happen:

\smallskip

\noindent\textbf{(i)} There exists a primitive idempotent $e$ with
$e^*=e$. In this case, let $\Delta=eAe$, which is a division
superalgebra with involution $-$ given by the restriction of $*$,
let $V=eA$, let $h:V\times V\rightarrow \Delta$ given by
$h(x,y)=xy^*$ for any $x,y\in V$, and let $\rho:A\rightarrow
\End_\Delta(V)$ be the map given by $x\mapsto R_x$ (the right
multiplication by $x$). It is clear that $h$ is an even
nondegenerate hermitian form. Note that for homogeneous $x=ea$,
$y=eb$ and $z=ec$ in $eA$ ($a,b,c\in A$),
$zh_{x,y}=h(z,x)y=zx^*y=za^*eb$, so $h_{V,V}=R_{AeA}$, which is
obviously contained in $R_A$. Hence, all the conditions in (i) are
satisfied.

\smallskip

\noindent\textbf{(ii)} $A_1=0$ and there exists a primitive
idempotent such that $eH(A,*)e^*=0$. In this case $eAe$ is a field
and there are elements $u\in eAe^*$ and $v\in e^*Ae$ such that
$u,v\in S(A,*)$, $uv=e$ and $vu=e^*$ hold. Consider here the field
$F=eAe$, let $V=eA$, and let $h:V\times V\rightarrow F$ given by
$h(x,y)=xy^*v$. Since $eH(A,*)e^*=0$, for any $x\in V$
$h(x,x)=(ex)(ex)^*v=exx^*e^*v\in eH(A,*)e^*v=0$, so $h$ is
alternating, and nondegenerate by primeness. With $\rho(x)=R_x$ as
before, the conditions in (ii) are satisfied.

\smallskip

\noindent\textbf{(iii)} There exists a primitive idempotent such
that $eA_0 e^*=0$ and two elements $u\in eA_1 e^*$ with $u^*=u$ and
$v\in e^*A_1 e$ with $v^*=-v$, such that $uv=e$ and $vu=e^*$. The
proof of Theorem \ref{th:primesupercases} shows that $eA_1 e=0$, so
consider the division algebra $D=eAe=eA_0 e$, the left vector space
$V=eA$ and the map $h:V\times V\rightarrow D$ given by
$h(x,y)=xy^*v$. Again, the proof of Theorem \ref{th:primesupercases}
shows that the map $-:D\rightarrow D$ given by $\bar d=ud^*v$ is an
involution, and then $h$ becomes an odd nondegenerate hermitian
form, because for homogeneous elements $x,y\in V$,
\[
\begin{split}
h(y,x)&=yx^*v=eyx^*v=u(vyx^*)v\\
  &=(-1)^{vy+vx+xy}u(xy^*v^*)v\\
  &=-(-1)^{x+y+xy}u(xy^*v)v\quad\text{(as $v^*=-v$ and $v$ is odd)}\\
  &=(-1)^{xy}u\overline{h(x,y)}
\end{split}
\]
as $h(x,y)=0=h(y,x)$ unless $x$ and $y$ have different parity. The
conditions in (iii) are satisfied here.

\medskip

\noindent Conversely, identifying $A$ with $\rho(A)$ in all the
cases, $A$ is a prime superalgebra with superinvolution because of
Corollary \ref{co:hVVsimple}. Now for any nonzero homogeneous
element $v\in V$, $I=h_{v,V}$ is a minimal right ideal of $A$ (Lemma
\ref{le:hvV}) generated by the primitive idempotent $e=h_{v,w}$,
where $w$ is an homogeneous element with $h(v,w)=1$, which satisfies
$e^*=(-1)^{vw}h_{w,v}$ (see the paragraph after Lemma \ref{le:hvV}).

\noindent Under the conditions of item (i), there is a nonzero
element $v\in V_0$ such that $h(v,v)\ne 0$, because otherwise, for
any $v,w\in V_0$, $h(v,w)=-h(w,v)=-\overline{h(v,w)}$, so the
restriction of the superinvolution $-$ to $\Delta_0$ would be minus
the identity, which is impossible. Now, if $v\in V_0$ and
$h(v,v)=d\ne 0$, then $h(w,v)=1$ with $w=d^{-1}v$. Besides $\bar
d=\overline{h(v,v)}=h(v,v)=d$, so with $e=h_{v,w}$,
$e^*=h_{v,w}^*=h_{w,v}:x\mapsto
h(x,d^{-1}v)v=h(x,v)d^{-1}v=xh_{v,w}$, and $e^*=e$. Thus $A$
satisfies the hypotheses of Theorem \ref{th:primesupercases}(i).

\noindent Under the conditions of item (ii), let $e=h_{v,w}$ with
$h(v,w)=1$, so $e^*=h_{w,v}$. Now, for any $f\in H(A,*)$,
\[
efe^*=h_{v,w}fh_{w,v}=h_{v,wf}h_{w,v}=h_{v,wfh_{w,v}}=h_{v,h(wf,w)v}=0,
\]
as $h(wf,w)=h(w,wf^*=h(w,wf)=-h(wf,w)$, since $h$ is alternating and
$f\in H(A,*)$. Thus $A$ satisfies the hypotheses of Theorem
\ref{th:primesupercases}(ii).

\noindent Finally, under the conditions of item (iii), take $v\in
V_0$ and $w\in V_1$ such that $h(v,w)=1$, and take the primitive
idempotent $e=h_{v,w}$. Then, for any $f\in\calL(V)_0$, $
efe^*=h_{v,h(wf,w)v}=0$, as $h(w,wf)\in h(V_1,V_1)=0$. Hence
$eAe^*=0$ and thus $A$ satisfies the hypotheses of Theorem
\ref{th:primesupercases}(iii).

\noindent Since the three possibilities in Theorem
\ref{th:primesupercases} are mutually exclusive, the same is valid
here.
\end{proof}

\begin{remark}\label{re:alternatingskewhermitian}
If the condition of $V_0\ne 0$ is omitted in item (i) of Theorem
\ref{th:primesuperinvolhermitian}, then item (i) would include the
situation of item (ii), as any alternating bilinear form on a vector
space $V$ over a field is an hermitian even form on the vector
superspace $V=V_1$.
\end{remark}

Because of Jacobson's Density (Corollary \ref{co:Jacobson}), if $A$
is a finite dimensional associative superalgebra and $V$ is an
irreducible right module for $A$, then $V$ is finite dimensional and
$A$ is isomorphic to $\End_\Delta(V)$, where $\Delta=\End_A(V)$ is a
finite dimensional division superalgebra. Hence:

\begin{corollary}\label{co:simplesuperinvolution} Let $A$ be a
finite dimensional simple superalgebra with $A_1\ne 0$, and let $*$
be a superinvolution of $A$. Then  one and only one of the following
situations occurs:
\begin{romanenumerate}
\item
There exists a finite dimensional division superalgebra with a
superinvolution $(\Delta,-)$, a finite dimensional left
$\Delta$-module $V$ with $V_0\ne 0$ endowed with a nondegenerate
hermitian even superform $h:V\times V \rightarrow \Delta$, and an
isomorphism of superalgebras with superinvolution $\rho:A\rightarrow
\End_{\Delta}(V)$ (the superinvolution on $\End_\Delta(V)$ is the
adjoint relative to the superform).
\item
 There exists a finite dimensional division algebra with an
involution $(D,-)$, a finite dimensional left $\bZ_2$-graded vector
space $V$ endowed with a nondegenerate hermitian odd form $h:V\times
V\rightarrow D$ , and an isomorphism of superalgebras with
superinvolution $\rho:A\rightarrow \End_D(V)$.
\end{romanenumerate}
\noindent Conversely, any such superalgebra is simple and it is
endowed with a superinvolution.
\end{corollary}

Notice that if $A$ is a CSS and it is isomorphic to $\End_\Delta(V)$
for some finite dimensional division superalgebra $\Delta$ and a
left vector space $V$ over $\Delta$, then the classes of $A$ and
$\Delta$ in the Brauer-Wall group coincide.

\section{Superinvolutions of the first kind}
\noindent In these sections all superantiautomorphisms and all
superinvolutions are $F$-linear.
\subsection{The odd case}

First of all, we consider the quadratic graded algebras.

\begin{lemma}\label{quadrsanti}
A quadratic graded algebra over a field $F$ has a
superantiautomorphism if and only if $-1$ is a square in $F$.
\end{lemma}

\begin{proof}
Let $K:=F \oplus Fu$, $u^2=a \in F^{\times}$, $\delta u=1$. We suppose that
there is an element $s \in F$ with $s^2=-1$. For $x:=\alpha1+\beta u
\in K$ ($\alpha, \beta \in F$) we define $\sigma(x)=\alpha1+\beta s
u$. The induced map $\sigma$ is a
superantiautomorphism of $K$.\\
On the other hand, if $\sigma$ is a superantiautomorphism, then for
all $\alpha, \beta \in F$ we have $\sigma(\alpha1+\beta
u)=\alpha1+\beta \lambda u$. The relation $\sigma(u^2)=-\sigma(u)
\sigma(u)$ implies that $\lambda^2=-1$.
\end{proof}

\begin{lemma}\label{quadrsinv}
Quadratic graded algebras do not posses superinvolutions.
\end{lemma}

\begin{proof}
We take a quadratic algebra $K:=F \oplus Fu$, $u^2=a \in F^{\times}$,
$\delta u=1$. If $\tau$ were a superinvolution, then
$\tau(u)=\lambda u$ for some $\lambda \in F^{\times}$. The fact that
$\tau^2(u)=u$ would imply that $\lambda=\pm 1$. Hence
$a=\tau(a)=\tau(u^2)=-(\tau(u))^2=-(\lambda u)^2=-a$, a
contradiction.
\end{proof}

\begin{remark}\label{quadralb}
If a quadratic algebra has a superantiautomorphism $\varphi$, then
$\varphi^2=\nu$, the grading automorphism.
\end{remark}

\noindent Now we can solve the case of algebras of odd type.

\begin{theorem}
Let A be a CSS of odd type over a field $F$. The algebra $A$ has a
superantiautomorphism if and only if the following two conditions
hold:\\
\textup{1.}\quad $-1$ is a square in $F$,\\
\textup{2.}\quad the algebra $A_0$ has an antiautomorphism.
\end{theorem}

\begin{proof}
If $A$ has a superantiautomorphism $\sigma$, then $\sigma\mid_{A_0}$
is an antiautomorphism of $A$. Moreover, the restriction
$\sigma\mid_{Z(A)}$ is a superantiautomorphism on the quadratic
graded algebra $Z(A)$. Hence Lemma \ref{quadrsanti} implies that
$-1$ is a square
in $F$. \\
Conversely, if $-1$ is a square in $F$ then, again by Proposition
\ref{quadrsanti}, there is a superantiautomorphism $\sigma_1$ on
$Z(A)$. If $\sigma_2$ is an antiautomorphism of $A_0$, then
$\sigma_1 \otimes \sigma_2$ is a superantiautomorphism of $A\simeq
(A_0)\otimes Z(A)$.
\end{proof}

\begin{theorem}\label{th:oddfirstkind}
Let A be a CSS of odd type. The algebra $A$ does not possess
superinvolutions of the first kind.
\end{theorem}

\begin{proof}
If $\tau$ were a superinvolution of $A$, then $\tau\mid_{Z(A)}$
would be a superinvolution on the quadratic graded algebra $Z(A)$, a
contradiction with Lemma \ref{quadrsinv}.
\end{proof}

\subsection{The even case}
\subsubsection{\bf The case $Z(A_0)\simeq F \times F$}
Let $A$ be a CSS of even type with $Z(A_0)\simeq F \times F$. Then
$A\simeq M_{n+m}(F) \wotimes (D)$, for a division algebra $D$. If
$D=F$, we are in the case of Proposition \ref{splitsuper}: if $A\sim
M_{n+m}(F)$, $n,m$ odd and $n \neq m$ there is a
superantiautomorphism but no superinvolution; otherwise $A$ has a
superantiautomorphism if and only if
it has a superinvolution.\\
It remains to consider the case where $D\neq F$.

\begin{theorem}
Let $D\neq F$. Then the following assertions are
equivalent:\begin{romanenumerate}
\item
$A$ has a superantiautomorphism
\item
$A$ has a superinvolution
\item
$D$ has an involution
\end{romanenumerate}
\end{theorem}

\begin{proof}
The superalgebra $A$ has a superantiautomorphism if and only if $A
\wotimes A \sim 1$ in $BW(F)$. Hence $(D) \wotimes M_{n+m}(F)
\wotimes M_{n+m}(F) \wotimes (D) \sim 1$ in $BW(F)$, also $(D)
\wotimes (D) \sim 1$ in $BW(F)$. But since the grading on $D$ is
trivial, this is equivalent with $D \otimes D \sim 1$ in $B(F)$. By
the classical Albert's Theorem, this is equivalent with the fact
that $D$ possess an involution. Hence (i)
and (iii) are equivalent.\\
To prove that (iii) implies (ii), suppose that $\sigma$ is an
involution on $D$. Let $V:=D^{n+m}$. If $e_1,\dots,e_{n+m}$ is the
canonical basis of $D^{n+m}$, we define a grading on $V$ by setting
$e_1,\dots,e_n \in V_0$ and $e_{n+1},\dots,e_m \in V_1$. Clearly, $A
\simeq \End_D(V)$. The fact that $D \neq F$ implies that $\{d \in D
\mid \sigma(d)=-d \}\neq \{0\}$ (if $\sigma$ were the identity on
$D$, then the central algebra $D$ would coincide with $F$). Let
$d\in D$ with $\sigma(d)=-d\neq 0$. The map $h:V\times V
\longrightarrow D$, given by $h(e_i,e_i)=1$ for all $i=1,\dots, n$,
$h(e_j,e_j)=d$ for all $j=n+1,\dots, n+m$, and $h(e_i,e_j)=0$ for
$i\ne j$, defines a superhermitian form and therefore there is an
associated superinvolution (see Corollary
\ref{co:simplesuperinvolution}). Therefore $(iii)$ implies
$(ii)$.\\
The implication $(ii)\Rightarrow (i)$ is trivial.\\
\end{proof}
\subsubsection{ \bf The case $Z(A_0)$ is a field}
We study CDS in relation with superinvolutions. We begin with graded
quaternion algebras.


\begin{lemma}\label{lemmaquatinv}
Let $Q=\langle a,b \rangle$ be a CDS. Then it has no superinvolution
(of the first kind).
\end{lemma}

\begin{proof}
Suppose that $A$ has a superinvolution $*$. Then there is a nonzero
element $u\in Q_1\cap H(Q,*)$ or $u\in Q_1\cap S(Q,*)$, because the
eigenvalues of the linear map $*$ are $\pm 1$ (or also: pick a
nonzero element $x$ in $Q_1$; if $x \notin S(Q,*)$, then take
$x+x^*\in H(Q,*)$ ). Let $u^2=\lambda \in F^{\times}$. Then
$\lambda=u^2=u^*u^*=-(u^2)^*=-(\lambda)^*=-\lambda$, which implies
that $\lambda=0$. Hence $u^2=0$, a contradiction to the fact that
$Q$ is a division superalgebra.
\end{proof}

\begin{lemma}\label{evendivision}
Let $\Delta$ be a CDS of even type, $\Delta_1\neq0$. Then $\Delta$
has no superinvolutions.
\end{lemma}

\begin{proof}
We prove the lemma by induction on the degree (as ungraded algebra)
of $\Delta$.
The case of degree $2$ is proved in Lemma \ref{lemmaquatinv}.\\
Now, suppose that the degree of $\Delta$ is greater than $2$ and
that $*$ is a superinvolution on $\Delta$. Since the eigenvalues of
the linear map $*$ are $\pm 1$, there is a nonzero element $u\in
\Delta_1\cap H(\Delta,*)$ or $u\in \Delta_1\cap S(\Delta,*)$. Let
$u^2=a$. Then $a^*=(u^2)^*=-(u^*)^2=-a$. Hence $a \notin F^{\times}$. On
the other hand, since $au=ua$, it follows that $a\notin
K:=Z(\Delta_0)$, because $Z(\Delta_0)=F1\oplus Fz$ and $xz=-zx$ for
all $x\in \Delta_1$. We define the map $\sigma:\Delta_0 \rightarrow
\Delta_0$, $\sigma(x)=uxu^{-1}$. Clearly, $\sigma^2(x)=axa^{-1}$.
Let $G:=\{x\in \Delta_0 \mid \sigma(x)=x\}.$ The algebra $G$ is a
division subalgebra of $\Delta_0$. The field $L=Z(G)$ contains $a$
and hence it is a proper extension of $F$. Now we define $\tilde
\Delta:=C_{\Delta}(L)$, the centralizer of $L$ in $\Delta$, a proper
division superalgebra
contained in $\Delta$.\\
We consider the irreducible representation of superalgebras
$\rho:\Delta \otimes L \rightarrow \End_F(\Delta)$, $d\otimes l
\mapsto \rho(d\otimes l)$, with $\rho(d\otimes l)(v):=lvd$ for all
$v \in \Delta$. Applying Jacobson's density (see Corollary
\ref{co:Jacobson}), we have
$$\Delta \otimes L \simeq \End( _{\tilde{\Delta}}\Delta)\simeq M_r(\tilde \Delta).$$
Therefore $L=Z(\tilde \Delta)$, where $\tilde \Delta$ is an even CDS
over $L$ with degree (as algebra over $L$) less than the degree of
the $F$ algebra $\Delta$. The map $*\otimes 1$ is a superinvolution
on the $L$-superalgebra  $\Delta \otimes L$. Hence also $M_r(\tilde
\Delta)$ has a superinvolution. It follows that the $L$-superalgebra
$\tilde \Delta$ has a superinvolution (see Corollary
\ref{co:simplesuperinvolution}$(i)$), a contradiction with the
induction hypothesis.
\end{proof}

\begin{theorem}\label{th:evenfirstkind}
Let $A$ be a CSS of even type and suppose that $Z(A_0)$ is a field.
Then it possesses no superinvolution.
\end{theorem}

\begin{proof}
We know that (see Remark \ref{superdiv})
$$A\simeq (M_n(F)) \wotimes \Delta,$$ where $\Delta$ is a
CDS. We observe, again by Corollary \ref{co:simplesuperinvolution}
that $A$ has a superinvolution if and only if $\Delta$ has a
superinvolution. Then the theorem follows immediately from Lemma
\ref{evendivision}.
\end{proof}

\begin{remark}
We observe that even if a CSS $A$ of even type where $Z(A_0)$ is a
field has no superinvolution, it may have order two in the
Brauer-Wall group. For example, if $-1$ is a square in $F$, then
every (even) superalgebra of the form $$F\langle \sqrt{a_1} \rangle
\wotimes \dots \wotimes F\langle \sqrt{a_{2r}}\rangle$$ has a order
two in the Brauer-Wall group because of Lemma \ref{quadrsanti}. In
particular, if $-1$ is a square in $F$, every division graded
quaternion algebra has order two in $BW(F)$ but has no
superinvolution.
\end{remark}

\subsection{An example: Clifford algebras}
Let $(V,q)$ be a quadratic space over the field $F$ and let $C(V,q)$
denote its Clifford algebra. The Clifford algebra is in a canonical
way a superalgebra: it possess the grading inherited from the
grading of the tensor algebra. Moreover, the Clifford algebra is
always endowed with a canonical involution, namely the involution
fixing all the elements of $V$. For more details about Clifford
algebras, see for example \cite{kmrt}. Thanks to our theorems, it is
very easy to establish the existence of superinvolutions on Clifford
algebras. It depends only on the dimension of $V$ (denoted by $\dim
q$) and the center of the even part of $C(V,q)$, denoted by
$Z(C_0)$.
\begin{corollary}
Let $(V,q)$ be a quadratic space over the field $F$. Then:
\begin{romanenumerate}
\item
If $\dim q$ is odd, then there exists no superinvolution.
\item
If $\dim q$ is even and $Z(C_0)$ is a field, then there exists no
superinvolution.
\item
If $\dim q$ is even and $Z(C_0)\simeq F \times F$, then there exists
a superinvolution.
\end{romanenumerate}
\end{corollary}

\section{The graded Albert theorem}
\noindent In this section we prove a graded version of Albert's
Theorem. Here we have a CSS with superantiautomorphism and want to
construct a superantiautomorphism whose square is the grading
automorphism. Again, all the superantiautomorphisms in this section
will be assumed to be $F$-linear.\\
In the first place, we define an invariant which will play an
important role in the proof of the graded Albert Theorem.

\begin{lemma}\label{invariant}
Let $A$ be a CSS of even type and suppose that the class of $A$ has
order $\leq 2$ in $BW(F)$. Let $\eta$ be a superantiautomorphism of
$A$. Then  there is an invertible even element $a$ such that
$\eta^2(x)=axa^{-1}$ for any $x\in A$. Moreover, $a\eta(a)=\eta(a)a
\in F^\times$ and the quantity $a\eta(a)F^{\times 2} \in
F^{\times}/F^{\times 2}$ does not depend on the choice of $\eta$.
\end{lemma}
\begin{proof}
Since $\eta^2$ is a graded automorphism which fixes elementwise
$Z(A_0)$, the graded Skolem-Noether theorem implies the existence of
a homogeneous element $a \in A_0$ such that $\eta^2(x)=axa^{-1}$ for
any $x$. We remark that $a\eta(a) \in F^\times$: this follows
directly from the relation
$\eta(\eta^2(x))=\eta^3(x)=\eta^2(\eta(x))$ valid for every $x\in
A$. Since $\eta^2$ fixes $F$, it follows that
$a\eta(a)=\eta(a)a \in F^\times$.\\
Now, let $\xi$ be another superantiautomorphism of $A$. Then there
is a homogeneous element $b\in A^\times$ such that for all $x\in A$
we have $\xi(x)=(-1)^{bx}b\eta(x)b^{-1}$. A computation shows that
$\xi^2(x)=b\eta(b)^{-1}axa^{-1}\eta(b)b^{-1}$. Let
$c:=b\eta(b)^{-1}a$. Then $c\xi(c)=a\eta(a)$. This number is
uniquely determined modulo squares because $a$ is uniquely
determined up to a nonzero scalar in $F$.
\end{proof}

\begin{remark}\label{invariantungraded}
In the ungraded case, $A$ is a central simple $F$-algebra with
involution of the first kind if and only if for any antiautomorphism
$\sigma$ of $A$ (with $\sigma^2=\iota_{a}$) we have $\sigma(a)a \in
F^{\times 2}$ (see \cite[p. 306]{scharl}).
\end{remark}

\begin{lemma}\label{albertgradeddiv}
Let $\Delta$ be a CDS of even type, $\Delta_1\neq 0$, and assume
that the class of $\Delta$ has order $\leq 2$ in $BW(F)$. Then
$\Delta$ has a superantiautomorphism $\varphi$ such that
$\varphi^2=\nu$, the grading automorphism.
\end{lemma}
\begin{proof}
Let $\eta$ be a superantiautomorphism of $\Delta$ and
$\eta^2(x)=axa^{-1}$, $a \in \Delta_0$.
Now, let
$Z(\Delta_0)=F\oplus Fz$, $z^2=f \in F^\times$ with $f\notin F^{\times2}$
and $\nu(x)=z^{-1}xz$. Moreover, we have $uz=-zu$ for all $u \in
\Delta_1$.\\
We may assume that $\eta\mid_{Z(\Delta_0)}=\mathrm{id}\mid_{Z(\Delta_0)}$
(we know that $\eta (z)=\pm z$: if $\eta(z)=-z$, we pick a nonzero element $u \in
\Delta_1$ and define the superantiautomorphism $\eta'$ by $\eta':=\iota_{u}\circ \eta$,
then $\eta'(z)=z$). \\
The map $\eta\mid_{\Delta_0}$ is a $Z(\Delta_0)-$linear
antiautomorphism, hence by the classical Albert's Theorem $\Delta_0$
has an involution. From Remark \ref{invariantungraded} it follows
that $a\eta(a) \in Z(\Delta_0)^{\times
2}$.\\
Let $\lambda \in Z(\Delta_0)^{\times}$ be such that $a\eta(a)=\lambda^2$.\\
We may assume that $a \in Z(\Delta_0)^{\times}$. In fact, if $a \notin
Z(\Delta_0)^{\times}$, then $1+\lambda^{-1}a\neq 0$. We define
$\xi=\iota_{(1+\lambda^{-1}a)^{-1}}\circ\eta$. Clearly,
$\xi\mid_{Z(\Delta_0)}=\mathrm{id}\mid_{Z(\Delta_0)}$ and
\begin{align}
\xi^2 &=\iota_{(1+\lambda^{-1}a)^{-1}}\circ \eta
\circ\iota_{(1+\lambda^{-1}a)^{-1}}\circ\eta \nonumber\\
&=\iota_{(1+\lambda^{-1}a)^{-1}}\circ\iota_{\eta(1+\lambda^{-1}a)}\circ\eta^2 \nonumber\\
&=\iota_{(1+\lambda^{-1}a)^{-1}\eta(1+\lambda^{-1}a)a} \nonumber\\
&=\iota_{(1+\lambda^{-1}a)^{-1}(a+\lambda^{-1}\eta(a)a)} \nonumber\\
&=\iota_{(1+\lambda^{-1}a)^{-1}(a+\lambda)} \nonumber\\
&=\iota_{\lambda} \nonumber
\end{align}
\noindent with $\lambda \in Z(\Delta_0)^{\times}$, as desired.\\
Since $a \in Z(\Delta_0)^{\times}$ and
$\eta\mid_{Z(\Delta_0)}=\mathrm{id}\mid_{Z(\Delta_0)}$, we have
$a\eta(a)=a^2 \in F^{\times}$. Hence $a\in F^{\times}$ or $a\in F^{\times}z$. But if
$a\in F^{\times}$, then $\eta^2=\mathrm{id}$, so it would be a
superinvolution, and this contradicts Lemma \ref{evendivision}. It
follows that $a\in F^{\times}z$. Therefore $\eta^2=\iota_{z}=\nu$, as
desired.
\end{proof}

\begin{theorem}\label{albertgraded}\textbf{\emph{(Graded Albert Theorem)}}\quad
Let A be a CSS which possesses a superantiautomorphism. Then $A$ has
also a superantiautomorphism $\varphi$ such that $\varphi^2=\nu$,
the grading automorphism.
\end{theorem}
\begin{proof}
If $A$ is odd, the existence of a superantiautomorphism implies that
the algebra $A_0$ has an involution $\sigma$ and that $\sqrt{-1}\in
F$. Since $\sqrt{-1}\in F$, then the quadratic graded algebra $Z(A)$
has a superantiautomorphism $s$ with $s^2=\nu$ (see Remark
\ref{quadralb}). Hence the odd superalgebra $A\simeq (A_0) \wotimes
Z(A)$ has the superantiautomorphism given by $\varphi=\sigma \otimes
s$
such that $\varphi^2=\nu$.\\
We follow the description of even CSS given in Remark
\ref{superdiv}.\\If $A$ is even and $A\simeq M_{r+s}(F) \wotimes
(D)$, then $D$ has an involution $\sigma$ and $M_{r+s}(F)$ has a
superantiautomorphism $\epsilon$ with $\epsilon^2=\nu$ (see
Proposition \ref{splitsuper}).
Hence the superantiautomorphism $\varphi=\epsilon \otimes \sigma$
has the property $\varphi^2=\nu$.\\
Finally, if $A$ is even and $A\simeq (M_n(F))\wotimes \Delta$, we
apply Lemma \ref{albertgradeddiv} to $\Delta$ to get a
superantiautomorphism $\epsilon$ with $\epsilon^2=\nu$. Tensoring
$\epsilon$ with the transpose, we get $\varphi:=t \otimes \epsilon$
with $\varphi^2=\nu$.
\end{proof}

\begin{corollary}
Let A be an even CSS which possesses a superantiautomorphism $\eta$
with $\eta^{2}=\iota_a$. Then $Z(A_0)\simeq F(\sqrt{\eta(a)a})$.
\end{corollary}
\begin{proof}
By Lemma \ref{invariant} and Theorem \ref{albertgraded} and the
proof of Lemma \ref{albertgradeddiv} we may assume that
$\eta^2=\nu=\iota_z$, with $Z(\Delta_0)=F\oplus Fz$, $z^2=f \in
F^\times$ with $f\notin F^{\times2}$ and $\eta(z)=z$. Then
$\eta(z)z=z^2=f$, and the result follows.
\end{proof}

\section{Superinvolutions of the second kind}
\noindent The situation for superinvolutions of the second kind, in
contrast with the case of superinvolutions of the first kind, is
analogous to the ungraded case in the sense that we can define a
corestriction and prove that a superinvolution of the second kind
exists if and only if the corestriction is trivial (the graded
Albert-Riehm Theorem). It is also interesting to observe that
superantiautomorphisms whose square is the grading automorphism do
not always exist if the corestriction is trivial.\\
We consider a CSS $A$ over a separable quadratic field extension
$K=F(\theta)$ with Galois group $\mathrm{Gal}(K/F)=\{1,j\}$ and
$\theta^2\in F\setminus F^2$, $\bar \theta:= j( \theta)=-\theta$.\\
A $K/F$-\emph{superantiautomorphism} $\xi$ of $A$ is a
superantiautomorphism which is $K/F$-semilinear, i.e. $\xi (\lambda
x)=\bar \lambda \xi(x)$ for all $\lambda \in K$, $x \in A$.
Accordingly, a $K/F$-superantiautomorphism $\xi$ is called
$K/F$-\emph{superinvolution} if $\xi^2(x)=x$ for all $x\in A$.
\begin{example}
Let $A:=M_{n+m}(K)$, $K=F(\theta)$. Then the CSS $A$ has always a
superinvolution of the second kind, namely the superinvolution
adjoint to the hermitian superform on $K^{n+m}$ given by the
diagonal matrix $\mathrm{diag}(1,\dots,1,\theta,\dots, \theta)$.
\end{example}
Let $\bar A$ be the superalgebra which is identical with $A^s$ as
ring but with the action of $K$ twisted by $j$. Let $T:=A \wotimes
\bar A$. We remark that $T$ is an even CSS over $K$. We consider the
map $\pi:T \rightarrow T$, defined in the following way: for all
homogeneous elements $p \in A$ and $q \in \bar A$ we have $p\otimes
q \mapsto (-1)^{pq} q \otimes p$. We now define
$$\mathrm{cor}_{K/F}(A)=\{ x \in T \mid \pi(x)=x\}.$$ Since
$T=\mathrm{cor}_{K/F}(A)+\theta \mathrm{cor}_{K/F}(A)\simeq
\mathrm{cor}_{K/F}(A)\otimes_F K$, we conclude that
$\mathrm{cor}_{K/F}(A)$ is an even CSS over $F$. Of course, if
$A=A_0$, the corestriction $\mathrm{cor}_{K/F}(A)$ coincides with
the usual ungraded corestriction (see \cite[p. 308]{scharl}).
\begin{example}\label{ex:corquad}
The corestriction of the quadratic graded algebra $A:=K\langle
\sqrt{\mu} \rangle$  (recall that $K\langle \sqrt{\mu} \rangle=K
\oplus Ku$ with the relation $u^2=\mu$ and the grading $\delta u=1$)
is the linear hull of the set $\{ 1 \wotimes 1, \theta u \wotimes u,
u\wotimes 1+ 1 \wotimes u, \theta u \wotimes 1 - 1 \wotimes \theta u
\}$. Since $\mathrm{cor}_{K/F}(A)\otimes_F K \simeq A \wotimes_K
\bar A$, which is a graded quaternion algebra over $K$, then
$\mathrm{cor}_{K/F}(A)$ is a graded quaternion algebra over $F$.
Here, $\mathrm{cor}_{K/F}(A)_0$ is the linear hull of $\{ 1 \wotimes
1, \theta u \wotimes u \}$. Now $(\theta u \wotimes u)^2=-(\theta
u)^2 \wotimes u^2=-\theta^2\mu \wotimes \mu=-\theta^2\mu \bar \mu (1
\wotimes 1)=\mathrm{N}_{K/F}(\theta \mu)(1 \wotimes 1)$. Hence
$\mathrm{cor}_{K/F}(A)_0\simeq F[\sqrt{\mathrm{N}_{K/F}(\theta
\mu)}]$
\end{example}
\noindent The next step is the definition of a (right)
$\mathrm{cor}_{K/F}(A)$-module structure on $A$ when $A$ possesses a
$K/F$-superantiautomorphism.\\
In particular, if $\xi$ is a $K/F$-superantiautomorphism, then we
define a right action of $T$ on $A$:
\[
\Xi: T \rightarrow \End_K(A),
\]
given by
\[
x\cdot(p\otimes q)=x(p\otimes q)^\Xi:=(-1)^{px}\xi(p)xq
\]
for all homogeneous $p\in A$,
$q\in \bar A$ and $x \in A$. The map $\Xi$
is an isomorphism (because $T$ is a CSS and by dimension count).\\
We will often use the equivalences \quad $A$ has a
$K/F$-superantiautomorphism  $\Leftrightarrow$
the $K$-CSSs $\bar A$ and $A^s$ are isomorphic $\Leftrightarrow$
$T \simeq \End_K(A) \sim 1$ in $BW(F)$.\\
Now $A$ is an irreducible (right) $T$-module. Since
$$\mathrm{cor}_{K/F}(A) \hookrightarrow T \simeq \End_K(A) \hookrightarrow \End_F(A),$$
we have a (right) $\mathrm{cor}_{K/F}(A)$-module structure on $A$.

\begin{lemma}\label{le:corquat}
Let $A$ be a CSS over $K$ with a $K/F$-superantiautomorphism $\xi$.
Let $C:=\End_{\mathrm{cor}_{K/F}(A)}(A)$. Then
$\mathrm{cor}_{K/F}(A) \sim C$ in $BW(F)$ and $\dim_F C=4$, $K
\subseteq C_0$. Moreover, $\mathrm{cor}_{K/F}(A) \sim 1$ if and only
if $A$ is not irreducible as $\mathrm{cor}_{K/F}(A)$-module.
\end{lemma}
\begin{proof}
The fact that $K \subseteq C_0$ follows immediately from
$\mathrm{cor}_{K/F}(A) \subseteq T \simeq \End_K(A)$.\\
First, we suppose that there exists a nontrivial proper irreducible
$\mathrm{cor}_{K/F}(A)$-submodule of $A$. Let $V\neq 0$ be an
irreducible $\mathrm{cor}_{K/F}(A)$-submodule of $A$ with $V\neq A$.
Then $KV:=V\oplus \theta V$ is a $K \otimes_F
\mathrm{cor}_{K/F}(A)\simeq T$-submodule, hence $KV=A$. By dimension
count, $\mathrm{cor}_{K/F}(A)\simeq \End_F(V)$. Therefore $A\simeq
V\oplus V$ (as $\mathrm{cor}_{K/F}(A)$-module). Hence we conclude
that $\End_{\mathrm{cor}_{K/F}(A)}(A)$ is a trivially graded split
quaternion algebra. This implies that $\mathrm{cor}_{K/F}(A) \sim 1$
in $BW(F)$. Now, let $A$ be an irreducible
$\mathrm{cor}_{K/F}(A)$-module. Then, by Schur's Lemma, $C$ is a CDS
and by Jacobson density we have that
$\mathrm{cor}_{K/F}(A)=\End_C(A)$. Since $(\dim_K A)^2=\dim_F
\mathrm{cor}_{K/F}(A)=\dim_F \End_C(A)=(\dim_C A)^2 \dim_F C$, we
have that $\dim_F C=(\frac{\dim_K A}{\dim_C A})^2=(\dim_K C)^2$.
This equation and the observation that $\dim_F C=(\dim_K C) (\dim_F
K)=2\dim_K C$ imply that $\dim_K C=2$ and hence $\dim_F C=4$. We
remark that if $C_1\neq 0$, then $K=C_0$.
\end{proof}

\begin{lemma}\label{le:xisquare}
Let $A$ be a CSS over $K$ with a $K/F$-superantiautomorphism $\xi$.
Then one and only one of the following cases occurs:
\begin{romanenumerate}
\item
Either $\xi^2=\iota_b$, with $b \in A_0^\times$. In this case
$\mathrm{cor}_{K/F}(A) \sim Q$ in $BW(F)$, where $Q=Q_0=K\oplus Ku$
is a quaternion algebra with $u^2=\xi(b)b\in F^\times$ and
$u\alpha=\bar\alpha u$ for any $\alpha\in K$,
\item
or $\xi^2=\iota_b$, with $b \in A_1^\times$. In this case
$\mathrm{cor}_{K/F}(A) \sim H$ in $BW(F)$, with $H$ a quaternion CDS
and $H_0=K$.
\end{romanenumerate}
\end{lemma}
\begin{proof}
The dichotomy is given by the observation that $\xi^2$ is a graded
automorphism and the Skolem-Noether Theorem. Recall from Lemma
\ref{le:corquat} that $A$ is a right module for
$\mathrm{cor}_{K/F}(A)$, and thus the action of the centralizer
$C:=\End_{\mathrm{cor}_{K/F}(A)}(A)$ will be written on the left.
Also, $K$ is a subalgebra of $C_0$, and the action (by scalar
multiplication) of any element $\alpha\in K$ will be denoted by
$l_\alpha$: $l_\alpha(a)=\alpha a$ for any $a\in A$. To prove the
lemma it is sufficient to compute $C$  (see Lemma \ref{le:corquat}).
We already know that $C$ is either a quaternion CDS or a quaternion
algebra with trivial grading.

\noindent Let $\xi^2=\iota_b$, with $b\in A_0^\times\sqcup
A_1^\times$. Note first that for any homogeneous $x\in A$,
$\xi^3(x)=\xi(\xi^2(x))=\xi\bigl((-1)^{bx}bxb^{-1}\bigr)
=(-1)^{bx}\xi(b)^{-1}\xi(x)\xi(b)$, but also
$\xi^3(x)=\xi^2(\xi(x))=(-1)^{bx}b\xi(x)b^{-1}$. Thus $\xi(b)b$ is
an even element in $Z(A)=K$, and it is fixed by $\xi$, so
$\xi(b)b\in F^\times$.

\noindent Consider the $F$-linear map $f:A\rightarrow A$ given by
$f(x)=(-1)^{xb}\xi(x)b$ for any homogeneous $x\in A$. The map $f$ is
even or odd, according to $b$ being even or odd. For any $\alpha\in
K$ and $x\in A$,
\[
f\circ l_\alpha(x)=f(\alpha(x))=\bar\alpha f(x)=l_{\bar\alpha}\circ
f(x),
\]
and
\[
f^2(x)=\xi\bigl(\xi(x)b\bigr)b=(-1)^{bx}\xi(b)\xi^2(x)b
=\xi(b)bxb^{-1}b=\xi(b)bx,
\]
so $f^2=l_{\xi(b)b}$ and the algebra over $F$ generated by $K$ and
$f$ (a subalgebra of $\End_F(A)$) is $K\oplus Kf$, which is a
quaternion (ungraded) algebra if $b$ is even, or a quaternion
superalgebra with even part $K$ if $b$ is odd. It remains to be
shown that $f$ lies in the centralizer $C$. For any homogeneous
elements $x,p,q\in A$,
\[
\begin{split}
f\bigl(x&\cdot(p\otimes q+(-1)^{pq}q\otimes p)\bigr)\\
  &=
  f\bigl((-1)^{px}\xi(p)xq +(-1)^{(p+x)q}\xi(q)xp\bigr)\\
  &=(-1)^{(p+x+q)b}\Bigl((-1)^{px}\xi\bigl(\xi(p)xq\bigr)b
     +(-1)^{(p+x)q}\xi\bigl(\xi(q)xp\bigr)b\Bigr)\\
  &=(-1)^{(p+x+q)b}\Bigl((-1)^{(p+x)q}\xi(q)\xi(x)\xi^2(p)b
      +(-1)^{px}\xi(p)\xi(x)\xi^2(q)b\Bigr)\\
  &=(-1)^{(p+x+q)b}\Bigl((-1)^{(p+x)q}(-1)^{pb}\xi(q)\xi(x)bp
      +(-1)^{px}(-1)^{qb}\xi(p)\xi(x)bq\Bigr)\\
  &=(-1)^{p(x+b)}\xi(p)\bigl((-1)^{xb}\xi(x)b\bigr)q +
       (-1)^{q(p+x+b)}\xi(q)\bigl((-1)^{xb}\xi(x)b\bigr)p\\
  &=f(x)\cdot (p\otimes q+(-1)^{pq}q\otimes p).
\end{split}
\]
Since the elements $p\otimes q+(-1)^{pq}q\otimes p$ span
$\mathrm{cor}_{K/F}(A)$, we conclude that $f$ is in the centralizer
$C$, as required.
\end{proof}

\begin{theorem}\label{th:gradedalbert}\textbf{\emph{(Graded Albert-Riehm Theorem)}}\quad
Let $K/F$ be a quadratic field extension. Let $A$ be a CSS over $K$.
Then $A$ has a $K/F$-superinvolution if and only if
$\mathrm{cor}_{K/F}(A) \sim 1$ in $BW(F)$.
\end{theorem}

\begin{proof}
First, we assume that $\xi$ is a $K/F$-superinvolution on $A$.
Clearly, $\xi^2=\iota_1$. By Lemma \ref{le:xisquare},
$\mathrm{cor}_{K/F}(A) \sim Q$, with $Q=K\oplus Ku$,
$u^2=\xi(1)1=1$, $Q=Q_0$. But $Q$ is split,
hence $\mathrm{cor}_{K/F}(A) \sim 1$ in $BW(F)$.\\
To prove the second implication, we suppose that
$\mathrm{cor}_{K/F}(A) \sim 1$ in $BW(F)$. Then
$K\otimes \mathrm{cor}_{K/F}(A) \simeq T \sim 1$ in
$BW(K)$. Hence $\bar A \simeq A^s$.\\
We know (see Theorem \ref{structureeven} and Remark \ref{superdiv})
that either $A$ is even and there is a central division $K$-algebra
$D$ such that $A\simeq M_{n+m}(K) \wotimes (D)$ (case I) or there
exists a $K$-CDS $\Delta$
(even if $A$ is even, odd otherwise) such that
$A\simeq (M_n(K))\wotimes \Delta$ (case II).\\
We remark that the CSSs $M_{n+m}(K)$ and $(M_n(K))$ have a
$K/F$-superinvolution. Hence, by the first implication,
they have trivial corestriction. \\
We consider case I. If $A\simeq M_{n+m}(K) \wotimes (D)$, then (see
\cite[p.~310]{scharl})
$$
\mathrm{cor}_{K/F}(A)\simeq \mathrm{cor}_{K/F}(M_{n+m}(K)) \wotimes \mathrm{cor}_{K/F}((D)).
$$
Hence in case I we have $\mathrm{cor}_{K/F}((D)) \sim 1$ in $BW(F)$
and since the grading of $D$ is trivial
the classical Albert-Riehm Theorem implies that $D$ has an
involution of the second kind.\\
To conclude the proof, we have to study case II. As in case one, we
obtain that $\mathrm{cor}_{K/F}(\Delta) \sim 1$ in $BW(F)$. Let
$\xi$ be a $K/F$-superantiautomorphism of $\Delta$ (it exists
because if $\mathrm{cor}_{K/F}(\Delta) \sim 1$ then $\bar \Delta
\simeq \Delta^s$). Since $\mathrm{cor}_{K/F}(\Delta) \sim 1$, Lemma
\ref{le:xisquare} forces $\xi^2=\iota_b$, with $b \in A_0^{\times}$.
Moreover, we may suppose that $\xi(b)b=1$ (because the fact that
$\mathrm{cor}_{K/F}(\Delta) \sim Q \sim 1$ implies that there is an
element $\lambda \in K$ such that $\lambda \bar \lambda= \xi(b)b$
and we may change $b$ with $\lambda^{-1}b$). If $b=-1$, we have
finished the proof. If $b \neq -1$, then the map
$\eta=\iota_{(1+b)^{-1}}\circ \xi$ is a superinvolution.
\end{proof}

\begin{example}
We consider the quadratic graded algebra $A:=K\langle \sqrt{\mu}
\rangle$ (with $K\langle \sqrt{\mu} \rangle=K \oplus Ku$ with the
relation $u^2=\mu$ and the grading $\delta u=1$). We have computed
the corestriction in Example \ref{ex:corquad}: we know that
$\mathrm{cor}_{K/F}(A) \sim 1$ in $BW(F)$ if and only if
$\mathrm{N}_{K/F}(\theta \mu)$ is a square in $F$. This condition is
equivalent with the existence of an odd element $v$ such that
$v^2\in F^{\times} \theta$ (or equivalently, to the existence of an element
$\alpha \in K^{\times}$ such that $\alpha^2\mu \in F^{\times}\theta$). In fact, if
$\alpha^2\mu=\gamma \theta$, $\gamma \in F^{\times}$, then
$\mathrm{N}_{K/F}(\theta \mu)=\mathrm{N}_{K/F}(\alpha^{-2} \gamma
\theta^{2})=\mathrm{N}_{K/F}(\alpha^{-1})^2(\gamma \theta^{2})^2\in
F^{\times 2}$. On the other hand, if $\mathrm{N}_{K/F}(\theta
\mu)=\gamma^2$,
 $\gamma \in F^{\times}$ we have $\mathrm{N}_{K/F}(\gamma^{-1}\theta\mu)=1$
 and by Hilbert 90 there is an element $\delta \in K^{\times}$ with
 $\gamma^{-1}\theta\mu=\delta \bar\delta^{-1}$. Hence
 $\bar\delta^2\mu=(\gamma\theta^{-2}\delta\bar\delta)\theta\in F^{\times}\theta$.
\end{example}

\begin{remark}
The preceding result can also be proved directly.
\begin{proposition}\label{oddsecond}
The CSS $A=K\langle \sqrt{\mu} \rangle$ has a $K/F$-superinvolution
if and only if there is an element $v \in A_1$ such that $v^2\in F^{\times}
\theta$.
\end{proposition}
\begin{proof}
Let $*$ be a $K/F$-superinvolution and $K\langle \sqrt{\mu}
\rangle=K \oplus Ku$.  Since the eigenvalues of $*$ are $\pm1,$ we
have $u^*=\pm u$. Then $\mu ^*=(u^2)^*=-u^* u^*=-\mu$. Hence $\mu
\in F^{\times} \theta$. Conversely, if there is an element $v \in A_1$ such
that $v^2\in F^{\times} \theta$, then the map $*:x+yv \mapsto \bar x +\bar
y v$ ($x,y \in K$) defines a $K/F$-superinvolution of $A$.
\end{proof}
\end{remark}
\noindent For CSS's of odd type there is a criterion for the
existence of involutions of the second kind which is easier than the
general one in terms of corestriction.
\begin{proposition}
Let $A$ be a CSS of odd type over $K$. Then $A$ has a
$K/F$-superinvolution if and only if the following two conditions
are satisfied:
\begin{romanenumerate}
\item
The CSS $Z(A)$ has a $K/F$-superinvolution.
\item
The central simple algebra $A_0$ has an involution of the second kind.
\end{romanenumerate}
\end{proposition}
\begin{proof}
We give two different proofs: one using Theorem
\ref{th:gradedalbert} and a direct proof.
First, we give the proof using Theorem \ref{th:gradedalbert},
which shows how the proposition is related to the general result.\\
If $Z(A)$ has a $K/F$-superinvolution and $A_0$ has an involution of
the second kind, then $\mathrm{cor}_{K/F}(Z(A)) \sim 1$ and
$\mathrm{cor}_{K/F}(A_0) \sim 1$. Hence
$$\mathrm{cor}_{K/F}(A)\simeq \mathrm{cor}_{K/F}(Z(A)) \wotimes
\mathrm{cor}_{K/F}(A_0)\sim 1,$$ as desired.
Conversely, suppose that $\mathrm{cor}_{K/F}(A) \sim 1$.
We know (see Lemma \ref{le:xisquare} and Proposition \ref{oddsecond})
that $\mathrm{cor}_{K/F}(Z(A)) \sim Q$, a quaternion algebra with
$Q=Q_0$, $K\subseteq Q$ and $\mathrm{cor}_{K/F}(A_0) \sim H$,
a graded quaternion algebra. We have only two possibilities:
either $Q\sim 1$ and $H\sim 1$ or $Q$ is a division algebra
and $H\simeq Q^s=Q^{op}$. This last case can not occur, because
$H_1\neq 0=Q_1^{op}$. Hence we conclude that
$\mathrm{cor}_{K/F}(A) \sim 1$ if and only if
$\mathrm{cor}_{K/F}(Z(A)) \sim 1$ and $\mathrm{cor}_{K/F}(A_0) \sim 1$.\\
Now we give an easier direct proof.\\
Clearly, if the $Z(A)$ has a $K/F$-superinvolution $\tau_1$ and  the
central simple algebra $A_0$ has an involution $\tau_2$ of the
second kind, then $\tau_1 \otimes \tau_2$ is a $K/F$-superinvolution
on $A=Z(A)\wotimes A_0$. On the other hand, if $\tau$ is a
$K/F$-superinvolution on $A$, then $\tau\mid_{A_0}$ is an involution
of the second kind on $A_0$ and $\tau\mid_{Z(A)}$ is a
$K/F$-superinvolution on $Z(A)$.
\end{proof}

\begin{example}
The existence of a superinvolution does not imply the existence  of
a superantiautomorphism whose square is the grading automorphism.
For example, consider the quadratic graded algebra $K\langle \sqrt i
\rangle=K \oplus Ku$ with $K=\mathbb{Q} (i)$, $i^2=-1$. Clearly, it
has a superinvolution. But suppose that $\varphi$ is a graded
superantiautomorphism with $\varphi^2=\nu$. Then $\varphi^2(u)=-u$.
On the other hand, there is a $\lambda \in K$ such that
$\varphi(u)=\lambda u$. Hence $-u=\varphi^2(u)=\varphi(\lambda
u)=\bar \lambda \lambda u$. But the equation $\bar \lambda
\lambda=-1$ has no solution in $K$, a contradiction to the existence
of $\varphi$.
\end{example}

\begin{example}
We give an example of even CSS with superinvolution but with  no
superantiautomorphism whose square is the grading automorphism. Let
$F=\mathbb{Q}$, $K=F (i)$, $i^2=-1$ and $A=K\langle
\sqrt{i}\rangle\wotimes K\langle \sqrt{3i}\rangle $. The CSS $A$ has
a superinvolution (and hence $\mathrm{cor}_{K/F}(A)\sim 1$), because
$K\langle \sqrt{i}\rangle$ and $K\langle \sqrt{3i}\rangle$ possess a
superinvolution. Here $A_0=K1\oplus Kz$ with $z^2=3$. Now, suppose
that $\varphi$ is a graded superantiautomorphism with
$\varphi^2=\nu$. Then $\varphi^2=\iota_z$ and $\varphi(z) \in A_0$
and $\varphi(z)^2=\varphi(z^2)=3.$ Moreover, $\varphi(z)=\pm z$. We
know (see Lemma \ref{le:xisquare}) that $\mathrm{cor}_{K/F}(A)\sim
K\oplus Kf$ with $f^2=\varphi(z)z=\pm z^2=\pm 3$. But $\pm 3 \notin
\mathrm{N}_{K/F}(K)$, hence $K\oplus Kf$ is a division algebra, a
contradiction with $\mathrm{cor}_{K/F}(A)\sim 1$.
\end{example}

\end{document}